\documentclass{article}
\usepackage{tikz}
\usepackage{tikz-cd}
\usepackage{amsmath}
\usepackage{amssymb}
\usepackage{amsthm}
\usepackage{authblk}
\usepackage{stmaryrd}
\usepackage{stackrel}
\usepackage{todonotes}
\usepackage{enumitem}
\usepackage[title]{appendix}
\usepackage{hyperref}
\usepackage[hyphenbreaks]{breakurl}
\usetikzlibrary{calc,decorations.pathreplacing,
	matrix,
	positioning,tikzmark}

\newcommand{\half}{\frac{1}{2}}
\newcommand{\ZZ}{\mathbb{Z}}
\newcommand{\into}{\hookrightarrow}

\newcommand{\Spec}{\mathrm{Spec}}
\newcommand{\naive}{\mathrm{Naive}}
\newcommand{\mc}{\mathcal}
\newcommand{\mono}{\rightarrowtail}

\newcommand{\iso}{\xrightarrow{\sim}}
\DeclareMathOperator*{\colim}{colim}

\DeclareMathOperator*{\hocolim}{hocolim}
\theoremstyle{plain}
\newtheorem*{theorem*}{Theorem}
\newtheorem{theorem}{Theorem}[section]
\newtheorem{corollary}[theorem]{Corollary}
\newtheorem{lemma}[theorem]{Lemma}
\theoremstyle{definition}
\newtheorem{definition}[theorem]{Definition}
\newtheorem{remark}[theorem]{Remark}

\numberwithin{equation}{section}

\title{Construction of the motivic cellular spectrum $\mathbf{KO}^{geo}$ over $Spec(\ZZ)$}

\author{K.~Arun Kumar}
\affil{\small Department of Mathematics, IISER Tirupati, India}

\date{}
\begin{document}
\maketitle
	\begin{abstract}
		We construct a periodic motivic spectrum over $Spec(\ZZ)$ which when pulled back to any scheme $S$ with $\half\in\Gamma(S,\mc{O}_S)$ is the $HP^1-$spectrum constructed by Panin and Walter. This spectrum $\mathbf{KO}^{geo}$ is constructed using closed subschemes of the Grassmannians $Gr(r,n)$. Using this we show that $\mathbf{KO}^{geo}$ is cellular. 
	\end{abstract}
	\section{Introduction}
	Throughout this paper all schemes we consider are separated and quasi-compact. Given any scheme $S$, we let $Sch_S$ be a small category equivalent to the category of $S$-schemes of finite type. After fixing $Sch_S$, the small categories $Sch^{aff}_S$, $Sm_S$ and $Sm^{aff}_S$ are the full subcategories of $Sch_S$ generated by (globally) affine, smooth and smooth affine $S$-schemes respectively. Let \emph{ind-scheme} refer to any presheaf on the categroy $Sm_S$ which is a directed colimit of representable presheaves. \par 
	Panin and Walter in \cite{PW18} construct an $HP^1$-spectrum $\mathbf{BO}$ over any regular Noetherian finite-dimensional scheme $S$ containing $\half$ and show that is it isomorphic to Hornbostel's hermitian $K$-theory spectrum $\mathbf{KO}$ in the stable motivic homotopy category $SH(S)$. Here $HP^1\simeq S^{4,2}$ is the quaternionic projective line. The main advantage of their construction is that $\mathbf{BO}_{2i}=\ZZ\times RGr$ and $\mathbf{BO}_{2i+1}=\ZZ\times HGr$ are ind-schemes $\ZZ\times RGr=\colim_{n} RGr(n,2n)$ and $\ZZ\times HGr=\colim_{n} HGr(2n,4n)$. Here $RGr(r,n)$ and $HGr(2r,2n)$ are open subschemes of a Grassmannian scheme of appropriate degree. This makes it easier to prove many properties. For example, R{\"o}ndigs and {\O}stv{\ae}r  in \cite{RO16} use this model to compute the slice spectral sequence of hermitian K-theory.\par 
	In this paper we remove the $\half\in\Gamma(S,\mc{O}_S)$ condition and extend the construction of $\mathbf{BO}$ to arbitrary schemes. We will denote this spectrum by $\mathbf{KO}^{geo}$ to be more in line with standard notations. Firstly for any $S$-scheme $X$, we define $KSp^\perp(X)$ to be the K-theory space associated to the symmetric monoidal category of unimodular alternating forms over $X$. This gives us a functor $KSp^\perp:Sm_S^{op}\to \mathbf{sSet}$ (cf.~Def.~\ref{def:KSp_perp}). Using the presheaf $KSp^\perp$ we can extend \cite[Thm.~8.2]{PW18} to get isomorphisms 
	\[ \ZZ\times HGr\cong \ZZ\times BSp_\infty\cong  R\Omega^1_sB(\coprod_n BSp_n) \cong KSp^\perp \]
	in the unstable motivic homotopy category $H_\bullet(S)$ over any scheme $S$. As the $\mathbb{A}^1$-invariance of orthogonal and symplectic K-theories are only known when $2$ is invertible, we still cannot extend the representability of hermitian K-theory \cite[Thm.~5.1]{PW18} to $\Spec(\ZZ)$. We however prove a weaker result for $KSp^\perp$ in section \ref{sec:unstable} using results from \cite{AHW18}.
	\begin{theorem*}
		Let $S$ be ind-smooth over a Dedekind ring with perfect residue fields. Then, for any affine $\Spec(R)\cong X\in Sm^{aff}_S$ there are isomorphisms
		\[[S^n\wedge X_+,\ZZ\times HGr]_{\mathbb{A}^1}\cong \pi_n Sing^{\mathbb{A}^1}(KSp^{\perp}(X)) \]
		for all $n\in \mathbb{N}$. In particular, this holds over $\Spec(\ZZ)$.
	\end{theorem*}
	 Stably, we are able to extend several results about $\mathbf{BO}$ to $\Spec(\ZZ)$. Collecting all the results from sections \ref{sec:geo_spec} and \ref{sec:prop} we get the following theorem.
	\begin{theorem*}
		For any scheme $S$, there exists a motivic cellular $HP^1$-spectrum \[\mathbf{KO}^{geo}_S=(\mathbf{KO}^{geo}_0,\mathbf{KO}^{geo}_1,\ldots)\in SH(S)_{HP^1}\cong SH(S)\] 
		such that,
		\begin{enumerate}
			\item $\mathbf{KO}^{geo}_{2n}\cong\ZZ\times RGr$ and $\mathbf{KO}^{geo}_{2n+1}\cong \ZZ\times HGr\cong KSp^\perp$ in $H_\bullet(S)$;
			\item $\Omega^2_{HP^1}\mathbf{KO}^{geo}\cong \mathbf{KO}^{geo}$ in $SH(S)_{HP^1}$ and hence $\Omega^4_{T}\mathbf{KO}^{geo}\cong \mathbf{KO}^{geo}$ as objects in $SH(S)$;
			\item for any morphism of schemes $f:S_1\to S_2$, there exists a canonical isomorphism $Lf^*\mathbf{KO}^{geo}_{S_2}\iso \mathbf{KO}^{geo}_{S_1}$ in $SH(S_1)$;
			\item if $f:S\to \Spec(\ZZ)$ is any scheme with $\half \in\Gamma(S,\mc{O}_S)$, $Lf^*\mathbf{KO}^{geo}$ is isomorphic to the motivic spectrum $\mathbf{BO}$ given in \cite{PW18}.
		\end{enumerate}
		In particular when $S$ is regular Noetherian of finite dimension with $\half\in\Gamma(S,\mc{O}_S)$, $\mathbf{KO}^{geo}$ represents hermitian K-theory.
	\end{theorem*}
	We do not know what cohomology theory $\mathbf{KO}^{geo}$ represents over $\Spec(\ZZ)$. The hope is that it represents some version of hermitian K-theory. In a recent paper Schlichting \cite{Sch19} introduced the notion of K-theory of forms which generalises the K-theory of spaces with duality. In this formalism $\mathbf{Symp}(X)$ becomes the category of quadratic spaces for a suitable choice of category with forms structure on vector bundles $Vect(X)$. If this theory satisfies Nisnevich excision and $\mathbb{A}^1$-invariance then $KSp$ will represent it in the unstable homotopy category. It is still unknown if this is true. However, a recent paper by Bachmann and Wickelgren (\cite{BW20}) suggests that there is a version of the hermitian K-theory ring spectrum which can be defined over arbitrary schemes (although it might not represent hermitian K-theory any more). It is unknown if this spectrum is stably equivalent to ours when $2$ is not invertible. 
	\section*{Acknowledgement}
		This paper is based on work done by me as part of my PhD and therefore it is adapted from my thesis \cite{kumar}. I would like to thank Universit\"at Osnabr\"uck for giving me the opportunity to do my PhD and supporting me as part of the graduate school ``DFG-GK1916 Combinatorial Structures in Geometry''. In particular I would like to thank my advisor Prof. Oliver R\"ondigs for guiding me throughout my PhD and for continuing to help me write this paper afterwards.  
	\section{Hermitian K-theory}\label{sec:HKT}
		Hermitian K-theory evolved out of the study of bilinear forms over rings and more generally schemes. Given a scheme $X$ and a quasicoherent $\mc{O}_X$-module $\mc{E}$, recall that a bilinear form on $\mc{E}$ is a morphism of $\mc{O}_X$-modules $\mc{E}\otimes\mc{E}\to \mc{O}_X$. We call a bilinear form \emph{unimodular} (sometimes referred to as non-degenerate) if the adjoint map $\mc{E}\to\mc{E}^* $ is an isomorphism. When $\mc{E}$ is isomorphic to a trivial vector bundle of rank $n$, each bilinear form can be represented by an element of $GL_n(\Gamma(X,\mc{O}_X))$. We will be interested in two classes of bilinear forms in particular. A \emph{symmetric bilinear form} is a bilinear form $\psi$ such that $\psi\tau=\psi$ where $\tau:\mc{E}\otimes\mc{E}\iso \mc{E}\otimes\mc{E}$ is the switch map, an \emph{alternating bilinear form} on $\mc{E}$ is a bilinear form  $\phi$, such that $\phi\circ\Delta=0$, where $\Delta$ is the diagonal map of sheaves $\Delta:\mc{E}_X\to\mc{E}\otimes\mc{E}$. We call a vector bundle equipped with a unimodular alternating form (resp. unimodular symmetric form) a symplectic space (resp. symmetric space). Over any scheme $X$, $H_+=(\mc{O}^{\oplus 2}_X, (\begin{smallmatrix}
		0 &1\\
		1 &0\\
		\end{smallmatrix}))$ and $H_-=(\mc{O}^{\oplus 2}_X,(\begin{smallmatrix} 
		0 &1\\
		-1 &0\\
		\end{smallmatrix}))$ are the hyperbolic symmetric and symplectic 
		spaces respectively. 
                Let $\mathbf{Symm}(X)$ and $\mathbf{Symp}(X)$ denote the categories of symmetric and symplectic spaces over a scheme $X$ respectively, where a morphism $f:(V,\phi)\to(W,\psi)$ is a morphism of vector bundles $f:V\to W$ such that $f^*\psi f=\phi$. The orthogonal sum $\perp$ turns $\mathbf{Symm}(X)$ and $\mathbf{Symp}(X)$ into (essentially small) symmetric monoidal categories. For objects $\mc{E}\in \mathbf{Symm}(X)$ and $\mc{F}\in \mathbf{Symp}(X)$ we denote their corresponding isomorphism groupoids by $O(\mc{E})$ and $Sp(\mc{F})$ respectively.
		\begin{definition}
			Let $R$ be a commutative ring.
			\begin{enumerate}
				\item The \textit{symplectic K-theory space} $KSp^\perp(R)$ is the space $K^{\perp}(i\mathbf{Symp}(\Spec(R)))$ and $KSp_n(R)= \pi_n KSp^\perp(R)$ are the symplectic K-groups.
			
				\item The \textit{orthogonal K-theory space} $KO^\perp(R)$ is the space $K^\perp(i\mathbf{Symm}(\Spec(R)))$ and $KO^\perp_n(R)=\pi_n KO^\perp(R)$ the orthogonal K-groups. 
			\end{enumerate}
		\end{definition}
		Here $K^\perp(-)$ is the K-theory space of a symmetric monoidal category \cite[Def.~4.3]{Wei13}. These were defined in \cite{Kar73} as ${}_1K^h(R)$ and ${}_{-1}K^h(R)$ respectively for rings where $2$ is invertible. The zeroth orthogonal K-group $KO^\perp_0(R)$ is equal to the classical Grothendieck-Witt group $GW(R)$ (called the \emph{Witt-Grothendieck group} in \cite[Def.~1.1]{Lam05}).\par  There exist monoidal functors 
		\[\coprod_n O(H^n_+)(R)\to \mathbf{Symm}(\Spec(R))\:\: \mathrm{and}\:\: \coprod_n Sp(H^n_-)(R)\to \mathbf{Symp}(\Spec(R)) \] 
		given on objects by $n\mapsto H^{n}_+$ and $n\mapsto H^{n}_-$ respectively. These then induce maps 
		\[\Omega_sB(\coprod_n BO(H^n_+))(R)\to KO^\perp(R)\] 
		\[\Omega_sB(\coprod_n BSp(H^n_-))(R)\to KSp^\perp(R)\] 
		of group completions. When $\half\in R$, \cite[Lem.~1.5]{Lam06} implies that $\coprod_n O(H^n_+)$ is a cofinal monoidal subcategory of $i\mathbf{Symm}(\Spec(R))$. We prove an analogous result for symplectic spaces below. 
        \begin{lemma}\label{lem:symplectic-cofinality}
        	Let $R$ be any ring. Every symplectic space over $R$ is isometric to a subspace of $H^n_-$ for some $n$.   
        \end{lemma}
        \begin{proof}
        	First note that $H^n_-$ is isometric to $(R^{2n},\begin{pmatrix}
        	0 &-I_{2n}\\
        	I_{2n} &0
        	\end{pmatrix} )$. Let $(P,\phi)$ be a symplectic space over $R$. As $P$ is projective, there exists $Q$ such that $P\oplus Q\cong R^m$ for some $m$. The symplectic space $P\perp P\perp H_-(Q)$, where $H_-(Q)=(Q\oplus Q^*,\begin{pmatrix}
        	0 &I_Q\\
        	-I_Q &0
        	\end{pmatrix} )$, then has an underlying space isomorphic to $R^{2m}$. Thus we have reduced to the case when the underlying module is free. Let $(R^{2m},S)$ be a symplectic space. As $S$ is an alternating invertible matrix, we have $S^{-1}=L-L^T$ for some strictly lower triangular matrix $L$. We will show that $(R^{4m},\begin{pmatrix}
        	S &0\\
        	0 &-S
        	\end{pmatrix})$ is isometric to $(R^{4m},\begin{pmatrix}
        	0 &-I_{2m}\\
        	I_{2m} &0
        	\end{pmatrix} )$. Consider the matrix $\begin{pmatrix}
        	L &I_{2m}\\
        	L^T &I_{2m}
        	\end{pmatrix}$.  Its transpose is $\begin{pmatrix}
        	L^T &L\\
        	I_{2m} &I_{2m}
        	\end{pmatrix}$ and hence we have
        	\begin{multline}
        	\begin{pmatrix}
        	L^T &L\\
        	I_{2m} &I_{2m}
        	\end{pmatrix}
        	\begin{pmatrix}
        	S &0\\
        	0 &-S
        	\end{pmatrix}
        	\begin{pmatrix}
        	L &I_{2m}\\
        	L^T &I_{2m}
        	\end{pmatrix}
        	=\begin{pmatrix}
        	L^T &L\\
        	I_{2m} &I_{2m}
        	\end{pmatrix}
        	\begin{pmatrix}
        	SL &S\\
        	-SL^T &-S
        	\end{pmatrix}\\=\begin{pmatrix}
        	L^TSL-LSL^T &-I_{2m}\\
        	I_{2m} & 0
        	\end{pmatrix}
        	\end{multline} 
        	As $L^T-L$ is the two sided inverse of $S$ we have
        	\[L^TSL-LSL^T=L^T(I_{2m}-SL^T)-LSL^T=L^T-L^TSL^T-LSL^T=L^T-L^T=0 \]
        	proving that $(R^{4m},\begin{pmatrix}
        	S &0\\
        	0 &-S
        	\end{pmatrix})\cong (R^{4m},\begin{pmatrix}
        	0 &-I_{2m}\\
        	I_{2m} &0
        	\end{pmatrix} )$.
        \end{proof}
		From this we have the following theorem.
		\begin{theorem}\label{th:symp_subspace}
			The morphism $\Omega_sB(\coprod_n BSp(H^n_-))(R)\to KSp^\perp(R)$ induces isomorphisms $\pi_n \Omega_sB(\coprod_n BSp(H^n_-))(R)\iso \pi_n KSp^\perp(R)$ for all $n\geq 2$. 
		\end{theorem}
		\begin{proof}
			This follows from the Cofinality Theorem \cite{GQ76} and Lemma~\ref{lem:symplectic-cofinality} above.
		\end{proof}
		There are several ways to extend the constructions of $KO^\perp(R)$ and $KSp^\perp(R)$ to arbitrary schemes. Firstly, we note that for any scheme $S$, $KO^\perp$ and $KSp^\perp$ define lax-functors $(Sch^{aff}_S)^{op}\to \mathbf{sSet}$ from affine $S$-schemes to simplicial sets. The naive way to extend this to arbitrary schemes is to use $\mathbf{Symp}(X)$ for non-affine schemes as well. For any $X\in Sm_S$, we define the categories of big symmetric and big symplectic spaces, $\mathbf{Symp}_{Sm_S}(X)$ and $\mathbf{Symm}_{Sm_S}(X)$ respectively, along the same lines as the category of big vector bundles (i,e. we fix a choice of pullback for each form)\cite[Sec.~10.5]{Wei13}. These are equivalent as symmetric monoidal categories to $\mathbf{Symp}(X)$ and $\mathbf{Symm}(X)$ respectively. We can then define simplicial presheaves
		\[ KO^{\perp}:Sm^{op}_S\to \mathbf{sSet}\]
		\[ KSp^{\perp}:Sm^{op}_S\to\mathbf{sSet} \]
		extending $KO^\perp(R)$ and $KSp(R)^\perp$ respectively.
		\begin{definition}\label{def:KSp_perp}
			Let $S$ be any scheme.	
			\begin{enumerate}
				\item $KO^{\perp}:Sm^{op}_S\to \mathbf{sSet}$ is the simplicial presheaf given by $KO^{\perp}(X)= K^{\perp}(i\mathbf{Symm}_{Sm_S}(X))$.
				\item $KSp^{\perp}:Sm^{op}_S\to \mathbf{sSet}$ is the simplicial presheaf given by $KSp^{\perp}(X)= K^{\perp}(i\mathbf{Symp}_{Sm_S}(X))$.
			\end{enumerate}
		\end{definition}
		For simplicity we will denote $Sp(H^n_-)$ by $Sp_{2n}$ from now on.
		\begin{theorem}\label{th:BSp_symp}
			The morphism $\coprod_n Sp_{2n}(-)\to i\mathbf{Symp}_{Sm^{op}_S}(-)$ in $Fun(Sm^{op}_S,\mathbf{Cat})$ given by $n\mapsto H^{n}_-$ induces a local weak equivalence of simplicial presheaves 
			\[\coprod_n BSp_{2n}\to B(i\mathbf{Symp}_{Sm^{op}_S})\] 
			with respect to the Zariski topology on $Sm_S$. In particular they are isomorphic as objects in $H_\bullet(S)$.
		\end{theorem}
		\begin{proof}
			Over a local ring $R$, every symplectic space is isometric to some $H^{ n}_-$ \cite[Thm.~5.8]{Lam06}. This implies that $\coprod Sp_{2n}(R)\to i\mathbf{Symp}(R)$ is an equivalence of groupoids and hence induces a weak equivalence of simplicial sets \[\coprod_nBSp_{2n}(R)\to Bi\mathbf{Symp}(R).\] As $BSp_{2n}$ is a degreewise representable simplicial sheaf, the stalks are just $\coprod_n BSp_{2n}(\mc{O}_{U,u})$, the evaluations at $\Spec(\mc{O}_{U,u})$. The isomorphism of simplicial sets, $\colim_{x\in U}Bi\mathbf{Symp}(U)\iso Bi\mathbf{Symp}(\mc{O}_{U,u})$, then induces a weak equivalence $\coprod_nBSp_{2n}(\mc{O}_{U,u})\to Bi\mathbf{Symp}(\mc{O}_{U,u})$ of stalks in the Zariski topology and so we are done. 
		\end{proof}
		This result implies that the induced map of objectwise group completions is also a weak equivalence.
		\begin{corollary}\label{th:BSp_symp2}
			For any scheme $S$, there are isomorphisms 
			\[\Omega^1_sB(\coprod_n BSp_{2n})\iso K^\perp(\coprod_n Sp_n)\iso KSp^\perp \]
			in $H_\bullet(S)$.
		\end{corollary}
		\begin{proof}
			Theorem \ref{th:BSp_symp} implies $KSp^\perp(\coprod_n Sp_n(\mc{O}_{U,u}))\iso KSp^\perp(\mc{O}_{U.u})$ for any point $(U,u)$. Therefore it is enough to show that these spaces are  weakly equivalent to the corresponding stalks. This follows from the analogous result for the classifying spaces and the construction of group completions given in \cite{GQ76}.
		\end{proof}
		\begin{remark}
			Currently there is some ambiguity regarding the correct definition of symplectic K-theory over an arbitrary scheme. In general the hermitian K-theory for categories with duality gives us Grothendieck-Witt spaces $GW(X)$ and $GW^-(X)$ (\cite{Sch10}) for any scheme $X$. These spaces give us all the desired properties for regular Noetherian schemes with $\half\in \Gamma(X,\mc{O}_X)$. In the general case, a recent paper \cite{9.herm1} constructs  multiple models of Grothendieck-Witt spaces (actually $\Omega$-spectra) which are equivalent when $2$ is invertible.  
		\end{remark}
		Let $X$ be a scheme, $i:U\into X$ be an open embedding and $n\in\mathbb{N}$. We denote by $KO^{[n]}(X)$, $KSp^{[n]}(X)$, $KO^{[n]}(X,U)$ and $KSp^{[n]}(X,U)$, the Grothendieck-Witt spaces as in \cite{Sch10b}.
		By \cite[Sec.~8 Cor.~1]{Sch10b} we have homotopy equivalences 
		\begin{equation}\label{spo_iso}
		KSp^{[n]}(X)\iso KO^{[n+4k+2]}(X) \quad\text{and}\quad KSp^{[n]}(X,U)\iso KO^{[n+4k+2]}(X,U)
		\end{equation} 
		for all $n,k\in\ZZ$ and for any open embedding $i:U\into X$.
 
		\begin{theorem}[{\cite[Thm.~5.1]{PW18}}]\label{ko:unstable_rep}
			Let $S$ be any regular Noetherian separated scheme of finite Krull dimension with $\half\in\Gamma(S,\mc{O}_S)$. For any $X\in Sm_S$ and any $n\geq 0$ there is an isomorphism of groups
			\begin{equation}
			KO^{[n]}_i(X)=\pi_i(KO^{[n]}(X))\cong [S^i\wedge X_+, KO^{[n]}]_{\mathbb{A}^1}
			\end{equation}
		\end{theorem} 
		In the case of affine schemes where two is invertible we get back $KSp^\perp$ and $KO^\perp$.
		\begin{theorem}
		Let $X\cong \Spec(R)$ and $\half\in R$, we then have homotopy equivalences
		\begin{equation}
		KO^{\perp}(X)\iso KO^{[0]}(X)
		\end{equation}
		\begin{equation}
		KSp^{\perp}(X)\iso KSp^{[0]}(X).
		\end{equation} 	
		\end{theorem}
		\begin{proof}
		When $\half\in R$, every skew-symmetric form is alternating and hence  $KO^\perp(X)$ and $KSp^{\perp}(X)$ are equal to ${}_1K^h(R)$ and ${}_{-1}K^h(R)$ of \cite{Kar73} respectively. The results \cite[Cor.~4.6]{Sch04} and \cite[Prop.~6]{Sch10b} supply the desired homotopy equivalences.
		\end{proof}
				
	\section{Unstable representability}\label{sec:unstable}
		The main reference for this section is \cite[Sec.~8]{PW18}. We fix a noetherian scheme $S$ of finite Krull dimension. The functor $X\mapsto Sp_{2n}(\Gamma(X,\mc{O}_X))$ is representable by a group scheme, in fact a closed subgroup scheme of $GL_{2n}$, which we also denote by $Sp_{2n}$. There are morphisms of schemes $Sp_{2n}\to Sp_{2n+2}$ given by $A\mapsto A\oplus\begin{pmatrix}
		0 &-1\\
		1 &0
		\end{pmatrix}$
		for each $n$. Let $Sp_\infty=\colim_n Sp_{2n}$ denote the colimit which is a group object in the category of motivic spaces. Let $HGr(2r,2n)$ be the quaternionic Grassmannian scheme which classifies rank $2r$ subbundles of $H^{2n}_-$ and $HGr$ is the ind-scheme $\colim_n HGr(2n,4n)$. 
		\begin{theorem}\label{iso_ksp}
			Over any noetherian scheme $S$ of finite Krull dimension we have a sequence of isomorphisms 
			\[\ZZ\times HGr\cong \ZZ\times BSp_\infty\cong  R\Omega^1_sB(\coprod_n BSp_{2n}) \cong KSp^\perp\] as objects in $H_\bullet(S)$.
		\end{theorem}
		This is equivalent to \cite[Thm.~8.2]{PW18} if  $2$ is invertible. We need the following lemma which is essentially \cite[Thm.~4.1]{MV99} but tweaked to correct a mistake initially pointed out in \cite{ST13}. Recall that a pointed graded simplicial sheaf of monoids is a quadruple $(M,+,\alpha,f)$ where $(M,+)$ is a sheaf of simplicial monoids and $\alpha:\mathbb{N}\to M$, $f:M\to\mathbb{N}$ are morphisms of sheaf of monoids with $f\alpha=Id$. In the lemma below by sheaf we mean sheaf over $Sm_S$ with the Nisnevich topology. 
	\begin{lemma}\label{lem:grp_completion}
		Let $(M,+,\alpha,f)$ be a pointed graded simplicial sheaf of monoids over the with $M_\infty=colim_n f^{-1}(n)$. Assume the following three conditions hold.
		\begin{enumerate}
			\item The map $\underline{\pi}^{\mathbb{A}^1}_0(f): \underline{\pi}^{\mathbb{A}^1}_0(M)\to\mathbb{N}$ is an isomorphism of constant sheaves.
			\item The monoid $(M,+)$ is commutative in $H_\bullet(S)$ under the induced monoidal structure.
			\item The diagram 
			\[
			\begin{tikzcd}
			M_{n}\times M_{n}\arrow[r,"+"]\arrow[d] &M_{2n}\arrow[d,"\alpha(2)+"]\\
			M_{n+1}\times M_{n+1}\arrow[r]  &M_{2n+2}
			\end{tikzcd}
			\]
			commutes in $H_\bullet(S)$.
		\end{enumerate}
		Then the canonical morphism $M_\infty\times\ZZ\to R\Omega^1_sBM$ is an $\mathbb{A}^1$-weak equivalence.  
	\end{lemma}
	Here $\underline{\pi}^{\mathbb{A}^1}_0(\mc{X})$ is just the sheafification of the presheaf of sets 
	\[U\mapsto [U,\mc{X}]_{\mathbb{A}^1}. \]
	\begin{proof}
		By \cite[Lem.~4.1.1]{MV99}, we can replace $M$ with a term wise free simplicial monoid and hence assume $M^+\cong R\Omega^1_sBM$.
		By \cite[Lem.~4.1.7]{MV99}, there is an $\mathbb{A}^1$-fibrant replacement functor $M\to Ex_{\mathbb{A}^1}(M)$ taking monoid objects to monoid objects. As $\underline{\pi}^{\mathbb{A}^1}_0(f): \underline{\pi}^{\mathbb{A}^1}_0(M)\to\mathbb{N}$ is an isomorphism, $Ex_{\mathbb{A}^1}(M)$ is also graded as \[\underline{\pi}^{\mathbb{A}^1}_0(Ex_{\mathbb{A}^1}(M))(U)\cong \pi_0(Ex_{\mathbb{A}^1}(M))(U).\]
		The morphism $M\to Ex_{\mathbb{A}^1}(M)$ induces an $\mathbb{A}^1$-weak equivalence of each graded component as they are disjoint and hence  $\mathbb{A}^1$-weak equivalences $R\Omega^1_sBM\iso R\Omega^1_s BEx_{\mathbb{A}^1}M$ and $M_\infty\iso Ex_{\mathbb{A}^1}(M)_\infty$ of homotopy colimits. Therefore, we can replace $M$ and $M_\infty$ with $Ex_{\mathbb{A}^1}(M)$ and $(Ex_{\mathbb{A}^1}(M))_\infty$ respectively and reduce to the situation where $M$ is $\mathbb{A}^1$-fibrant. Hence we can assume $(M,+,0)$ is commutative in the simplicial homotopy category $H^s(Sm_S)$ and the diagram in (3) commutes up to simplicial homotopy. Now we need to show that $M_\infty\times\ZZ\to M^+$ is a Nisnevich local weak equivalence of simplicial sheaves. As the Nisnevich site has enough points we use the stalk functors to reduce to the case where all objects are Kan complexes. The first two conditions then imply that the map $M_\infty\times\ZZ\to M^+$ , where $M^+$ is the group completion of the simplicial monoid $M$, is a homology isomorphism. Condition (3) implies that $M_\infty$ is an H-space and therefore $\pi_1(M_\infty)$ is abelian and acts trivially on all higher homotopy groups. The map is then a weak equivalence by Whitehead's theorem.
	\end{proof}
	\begin{remark}
		Condition 3 was not part of \cite[Thm.~4.1.10]{MV99}. As the counter-example in \cite[Remark 8.5]{ST13} shows, this additional condition is necessary.
	\end{remark}
	\begin{proof}[Proof of Theorem~\ref{iso_ksp}]
		Firstly by Corollary~\ref{th:BSp_symp2} we have $R\Omega^1_sB(\coprod_nBSp_{2n})\iso KSp^\perp$ in $H_\bullet(S)$. We will show that the graded sheaf of monoids $\coprod_nBSp_{2n}$ satisfies the conditions of Lemma~\ref{lem:grp_completion}. Condition (1) is clear. For (2) we need $\coprod_nBSp_{2n}$ to be commutative in $H_\bullet(S)$. Let $\Delta R$ be the simplicial ring with $\Delta R_n = R[t_0,\ldots,t_n]/(t_0+\ldots+t_n-1)$ and structure maps same as the topological simplex. Taking stalks we can reduce to the case of showing the simplicial monoid $\coprod_n Sing^{\mathbb{A}^1}(BSp_{2n})(R)=\coprod_n BSp_{2n}(\Delta R)$ is homotopy commutative for any ring $R$. Fixing $n,m\in\mathbb{N}$, there exists a permutation matrix $P_{n,m}\in GL_{2n+2m}(R)$ such that, \[\begin{pmatrix}
		A &0\\
		0 &B
		\end{pmatrix} =P_{n,m} \begin{pmatrix}
			B &0\\
			0 &A
		\end{pmatrix} P^{-1}_{n,m}\] 
		in $Sp_{2n+2m}(R)$ for any $A\in Sp_{2n}(R)$ and $B\in Sp_{2m}(R)$. From \cite[III.1.2.1]{Wei13} we see that $P_{n,m}$ is a product of elementary matrices $P_{n,m}=E_1\ldots E_k$. As these block matrices are of even rank, we can choose $E_i\in  Sp_{2n+2m}(R)$. Each of these $E_i$ can be written as the sum of the identity matrix and a nilpotent matrix $E_i=I+N_i$. The matrix $F=I+tN_i$ is then an element of $Sp_{2n+2m}(R[t])$ and gives us a path between $E_i$ and $I$ in $Sp_{2n+2m}(\Delta R)$.  Therefore, there is a simplicial homotopy between the maps $(A,B)\mapsto A\oplus B$ and $(A,B)\mapsto B\oplus A$ as functors $Sp_{2n}(\Delta R)\times Sp_{2m}(\Delta R)\to Sp_{2n+2m}(\Delta R)$. For (3) taking stalks again we need to show that for any ring $R$ and any $n$ the maps $Sp_{2n}(\Delta R)\times Sp_{2n}(\Delta R)\to Sp_{4n+4}(\Delta R)$ given by 
		\[(A,B)\mapsto A\oplus B\oplus \begin{pmatrix}
		0 &-1\\
		1 &0
		\end{pmatrix}\oplus \begin{pmatrix}
		0 &-1\\
		1 &0
		\end{pmatrix}\] and 
		\[
		(A,B)\mapsto A\oplus \begin{pmatrix}
		0 &-1\\
		1 &0
		\end{pmatrix}\oplus B\begin{pmatrix}
		0 &-1\\
		1 &0
		\end{pmatrix} \] 
		are homotopic. Like (2) these maps are equal up to conjugation by a permutation matrix and are simplicially homotopic.  Hence we get an isomorphism $\ZZ\times BSp_\infty\iso R\Omega^1_sB(\coprod_n BSp_{2n})$ in $H_\bullet(S)$. Finally, we have $HGr(2n,\infty)\iso BSp_{2n}$ from \cite[Sec 8.]{PW18} inducing $\ZZ\times HGr\iso \ZZ\times BSp_\infty$ in $H_\bullet(S)$. Note that Panin and Walter assume the underlying scheme is regular with $2$ invertible but the proof works for arbitrary schemes.
	\end{proof}
	 The above proof also shows that $\ZZ\times HGr$ is a unital commutative monoid object in $H_\bullet(S)$. Using results from \cite{AHW18} we can say more in the case of affine schemes.
	 \begin{theorem}\label{th:affKsp0}
	 	Let $S$ be ind-smooth over a Dedekind ring with perfect residue fields. Then, for any $X\in Sm^{aff}_S$ there is an isomorphism, 
	 	\[[X_+,\ZZ\times HGr]_{\mathbb{A}^1}\cong KSp^\perp_0(X) \]
	 	of groups. 
	 \end{theorem}
 	\begin{proof}
 		By \cite[Ex.~2.3.4]{AHW18} we have for each $n$, bijections of sets
 		\[[X_+, BSp_{2n}]_{\mathbb{A}^1}\cong \pi_0(B_{Nis}Sp_{2n})(X)\cong Symp_n(X) \]
 		where $Symp_n(X)$ is the set of isometry classes of rank $2n$ symplectic spaces over $X$ and $B_{Nis}Sp_{2n}$ is a Nisnevich fibrant replacement of $BSp_{2n}$. As the filtered colimit $BSp_\infty=\colim_{n}BSp_{2n}$ is also the homotopy colimit, the sheaf $\underline{\ZZ}\times \colim_n B_{Nis}Sp_{2n}$ is a Nisnevich fibrant replacement of $\ZZ\times BSp_{\infty}$ giving,
 		\[[X_+, \ZZ\times BSp_{\infty}]_{\mathbb{A}^1}\cong\underline{\ZZ}(X)\times\colim_{n} \pi_0(B_{Nis}Sp_{2n}(X))\cong \underline{\ZZ}(X)\times \colim_{n}Symp_n(X) \]
 		as sets. There is a map $\underline{\ZZ}(X)\times \colim_{n}Symp_n(X)\to KSp^\perp_0(X)$ which is given by \[(i,[A])\mapsto [A]-(\frac{rank(A)}{2}-i)[H_-]\] on each connected open subscheme of $X$. As the monoid structure on $\ZZ\times BSp_\infty$ is induced by $\perp$, this is a monoid homomorphism. We will show that this map is a bijection. First we note that it is enough to prove this in the case when $X$ is connected as all our schemes are locally connected. Given any $[A]-[B]\in KSp^\perp_0(X)$, by Theorem~\ref{th:symp_subspace}, there exists a symplectic space $C$ such that $B\perp C\cong H^k_-$ for some $k$. Then, $[A]-[B]=[A]+[C]-(k[H_-])=[A\perp C]-k[H_-]$ and hence the map is surjective. Now suppose $[A]-k[H_-]=[B]-j[H_-]$ in $KSp^\perp(X)$. Then, $A\perp H^{j+p}_-\cong B\perp H_-^{k+p}$ for some $p$. Hence they have the same preimage in $\ZZ\times \colim_{n}Symp_n(X)$ and this is a group isomorphism.    
 	\end{proof}
	 \begin{theorem}\label{th:affKsp}
	 	Let $S$ be ind-smooth over a Dedekind ring with perfect residue fields. Then, for any affine $\Spec(R)\cong X\in Sm^{aff}_S$ there are isomorphisms
	 		\[[S^n\wedge X_+,\ZZ\times HGr]_{\mathbb{A}^1}\cong \pi_n Sing^{\mathbb{A}^1}(KSp^{\perp}(X)) \]
	 		for all $n\geq 0$. In particular, this holds over $\Spec(\ZZ)$.
 	\end{theorem} 
 	\begin{proof}
 		By \cite[Thm.~4.1.2]{AHW18}, $B_{Nis}Sp_{2k}$ is $\mathbb{A}^1$-naive and by \cite[Prop.~4.1.16]{MV99} the map $BSp_{2k}\to B_{Nis}Sp_{2k}$ induces an isomorphism of groups $\pi_nBSp_{2k}(U)\iso\pi_kB_{Nis}Sp_{2k}(U)$ for all $U\in Sm_S$ and all $n\geq 1 $. Therefore, the morphism of colimtis $BSp_{\infty}\to \colim_{k} B_{Nis}Sp_{2k}$ also induces an isomorphism on all higher homotopy groups. By Theorem~\ref{th:affKsp0}, we have a group isomorphism $\pi_0\underline{\ZZ}\times \colim_{k} B_{Nis}Sp_{2k}(X)\cong KSp_0(X)$ for every $X\in Sm^{aff}_S$. As $\pi_0BG=*$ for any group, there is a weak equivalence $KSp^\perp_0(X)\times BSp_\infty(X)\to \underline{\ZZ}(X)\times \colim_{k}B_{Nis}Sp_{2k}(X)$. We then have bijections,
 		\begin{multline*}
 		[S^n\wedge X_+,\ZZ\times BSp_\infty]_{\mathbb{A}^1}\cong [S^n\wedge X_+,\underline{\ZZ}\times \colim_k B_{Nis}Sp_{2k}]_{\mathbb{A}^1}\\\cong \pi_n Sing^{\mathbb{A}^1}(\underline{\ZZ}(X)\times \colim_kB_{Nis}Sp_{2k}(X) ) 
 		\end{multline*}
 		for all $n\geq 0$. The map of spaces $KSp^\perp_0(X)\times BSp_\infty(X)\to KSp^\perp(X)$ is a topological group completion, when $X$ is affine, by Theorem~\ref{th:symp_subspace}. Further, from the proof of Theorem~\ref{iso_ksp} we have that $Sing^{\mathbb{A}^1}BSp_\infty(X)$ is a grouplike H-space. Putting these together we get that the map $Sing^{\mathbb{A}^1} (KSp^\perp_0(X)\times BSp_\infty(X))\to Sing^{\mathbb{A}^1}(KSp^\perp(X))$ is a levelwise weak equivalence. Therefore,
 		\[\pi_nSing^{\mathbb{A}^1}(\underline{\ZZ}\times \colim_kB_{Nis}Sp_{2k})\cong \pi_nSing^{\mathbb{A}^1}(KSp^\perp(X)) \]
 		and hence we have
 		\[[S^n\wedge X_+,\ZZ\times BSp_\infty]_{\mathbb{A}^1}\cong \pi_nSing^{\mathbb{A}^1}(KSp^\perp(X)) \]
 		thus completing the proof.
 	\end{proof}
	We have a stronger result in the case when $\half\in\Gamma(S,\mc{O}_S)$.
	\begin{theorem}
		Let $S$ be a regular Noetherian scheme of finite Krull dimension with $\half\in\Gamma(S,\mc{O}_S)$. For any $X\in Sm_S$ and for all $n\in\mathbb{N}$,
		\[ [S^n\wedge X_+,\ZZ\times HGr]_{\mathbb{A}^1}\cong KSp^{[0]}_n(X).\]
	\end{theorem}	
	\begin{proof}
          The statement follows from Theorems~\ref{ko:unstable_rep} and \ref{iso_ksp}.
        \end{proof}
        
	\section{Hermitian K-theory spectrum} 
	Recall from Section~\ref{sec:HKT} that for any scheme $S$ we have simplicial presheaves $KO^{[n]}:Sm^{op}_S\to \mathbf{sSet}$. When $S$ is a regular Noetherian scheme of finite Krull dimension with $\half\in\Gamma(\mc{O}_S,S)$, Panin and Walter showed that there is a motivic $T$-spectrum $\mathbf{KO}=(KO^{[0]},KO^{[1]},KO^{[2]},\ldots)$, where $T=\mathbb{A}^1/\mathbb{A}^1-0$ \cite[Sec. 7]{PW18}. We will recall this construction below. First, we need the fact that for any exact category with strict duality $(\mc{E},*,\eta)$ we can define a family of exact category with weak equivalence and duality structures on the category of bounded chain complexes $Ch^b(\mc{E})$. 
	\begin{definition}[Shifted dualities]
		For $n\in\ZZ$, we define the functor $\eta^{n}:Ch^b(\mc{E})\to Ch^b(\mc{E})^{op}$ to be given by $E.\mapsto (E.)^*[n]$. Where $(E.)^*$ is the chain complex $(E.)_i^*=E^*_{-i}$. That is,
		\[(E^*,d^*):\quad \ldots E_{-i+1}^*\xrightarrow{d^*_{-i}} E^*_{-i}\xrightarrow{d^*_{-i-1}} E^*_{-i-1}\to\ldots  \]
		$[n]$ is the usual shift functor $E[n]_i=E_{i-n}$. In particular $(d^{*^n})_i=(d_{-i-1-n})^*$. Let $\eta^{n}:(-)\Rightarrow(-)^{*^{n}*^{n}}$ be the natural transformation given by \[(\eta_{E}^n)_i=(-1)^{\frac{n(n-1)}{2}}\eta_{E_i}\quad E.\in Ch^b(\mc{E})\]
	\end{definition}
	The pairs $(*^{n},\eta^{n})$, give us exact categories with weak equivalences and duality
	\[(Ch^b(\mc{E}),*^n,\eta^n,q)\]  
	for each $n\in\ZZ$. We will apply this construction to $\mc{E}=Vect(X)$ and $*^n=\vee^n$ given by $\vee^n:E\mapsto E^\vee[n]$.
	\begin{definition}[Koszul complex]
		Let $p:E\to X$ be a vector bundle of rank $n$. The pullback $p^*E=E\times_X E\to E$ is a vector bundle over $E$ (as a scheme) and has a section $s:E\to E\times_X E$ given by the diagonal map. The \emph{Koszul complex} $\kappa(E)$ is the chain complex of vector bundles over $E$ given by
		\[\kappa(E):(0\to \Lambda^np^*E^\vee\to \Lambda^{n-1}p^*E^\vee\to\ldots\to\Lambda^2p^*E^\vee\to p^*E^\vee\to\mc{O}_E\to 0) \]
		with grading $\kappa(E)_i=\Lambda^{n-i}E^\vee$ and differentials $d:\Lambda^{k+1}p^*E^\vee\to \Lambda^kp^*E^\vee$ given by 
		\[d(x_0\wedge x_1\wedge\ldots\wedge x_k)=\sum^{n}_{i=0}(-1)^i s^*(x_i)x_0\wedge\ldots\wedge\hat{x}_i\wedge\ldots\wedge x_k  \] 
		where $s^*$ is the dual of the section $s:\mc{O}_E\to p^*E$.
	\end{definition}  
	The canonical isomorphism $\Lambda^r p^*E\iso (\Lambda^{n-r} p^*E)^\vee\otimes\Lambda^n p^*E$ induces an isomorphism of chain complexes $\theta(E):\kappa(E)\iso \kappa(E)^\vee\otimes\Lambda^n p^*E[n]$. Given an isomorphism $\lambda:det E=\Lambda^nE\iso \mc{O}_X$, this gives us a non-degenerate symmetric form in $Ch^b(Vect(E),\vee^n,\eta^n,q)$,
	\[\kappa(E,\lambda):\kappa(E)\iso \kappa(E)^\vee\otimes\Lambda^n p^*E[n]\iso \kappa(E)^\vee[n] \]
	where we choose the sign of the isomorphisms $\Lambda^r p^*E\iso (\Lambda^{n-r} p^*E)^\vee\otimes\Lambda^n p^*E$ to be compatible with the natural transformation $\eta^n$. Let $Vect(E)^{q_{E-X}}$ be the full subcategory of $Vect(E)$ generated by complexes which are acyclic when restricted to the open subscheme $E-X$. It follows that $\kappa(E,\lambda)$ is an object of $Vect(E)^{q_{E-X}}$.  
	Given a pair $(E,\lambda)$, where $p:E\to X$ is a rank $n$ vector bundle and $\lambda:det E=\Lambda^n E\iso \mc{O}_X$ an isomorphism between the determinant bundle and the trivial bundle, the Thom class $th(E,\lambda)=[(\kappa(E,\lambda)]$ is the corresponding element in $KO^{[n]}_0(E,E-X)$. 
	From this we see that for every pair $(E,\lambda)$, where $E$ is a vector bundle and $\lambda:\Lambda^n E\iso \mc{O}_X$, we have a functor
	\[Ch^b(Vect(X),\vee^m,\eta^n,q)\to Ch^b(Vect(E)^{q_{E-X}},\vee^{m+n},\eta^n,q)\] 
	which on objects is given by tensoring with $\kappa(E)$,
	\[C.\mapsto p^*C.\otimes \kappa(E).\]
	This induces a map of spaces $KO^{[m]}(X)\to KO^{[m+n]}(E,E-X)$ which we denote by $\otimes th(E,\lambda)$. 
	\begin{theorem}\label{th:thom_iso}
		Let $X\in Sm_S$ with $S$ a regular Noetherian scheme of finite Krull dimension with $\half\in\Gamma(S,\mc{O}_S)$. For any pair $(E,\lambda)$ described above, the map $\otimes th(E,\lambda):KO^{[m]}(X)\to KO^{[m+n]}(E,E-X)$ is a weak equivalence of spaces.
	\end{theorem}
	\begin{proof}
          This is a corollary of \cite[Thm.~5.1]{PW18}, which contains the
          hypotheses on regularity and invertibility of $2$.
        \end{proof}

        For any scheme $X$, let $E=\mc{O}_X$ be the trivial bundle with $\lambda=id$, $\kappa(E,\lambda)$ is then given by,
	\[
	\begin{tikzcd}
	0\arrow[r] &\mc{O}_{\mathbb{A}_X^1}\arrow[r,"t"]\arrow[d,"-1"] &\mc{O}_{\mathbb{A}_X^1}\arrow[r]\arrow[d,"1"] &0\\
	0\arrow[r] &\mc{O}_{\mathbb{A}_X^1}\arrow[r,"-t"] &\mc{O}_{\mathbb{A}_X^1}\arrow[r] &0
	\end{tikzcd}
	\] 
	where $t$ is the variable in $\mathbb{A}_{\Spec(R)}^1\cong \Spec(R[t])$. We then have an induced equivalence of simplicial sets,
	\[
	\otimes th(\mc{O},id):KO^{[n]}(X)\iso KO^{[n+1]}(\mathbb{A}^1\times X,(\mathbb{A}^1\setminus \{0\})\times X)
	\]
	which is functorial on $X$. Therefore, there is a levelwise weak equivalence of simplicial presheaves $KO^{[n]}(-)\iso KO^{[n+1]}(\mathbb{A}^1\times-,(\mathbb{A}^1\setminus \{0\})\times-)$. If $KO^{[n]}_f$ is an $\mathbb{A}^1$-fibrant replacement of $KO^{[n]}$, we have a zigzag of $\mathbb{A}^1$-weak equivalences
	\begin{multline*}
	\mathbf{Hom}_{sPSh_\bullet(S)}(-\wedge T,KO^{[n]}_f)\xleftarrow{\sim} \mathbf{Hom}_{sPSh_\bullet(S)}(-\wedge T,KO^{[n]})\\ \iso KO^{[n]}(\mathbb{A}^1\times-,(\mathbb{A}^1\setminus \{0\})\times-)
	\end{multline*}
	for each $n$. As $\mathbf{Hom}_{sPSh_\bullet(S)}(-\wedge T,KO^{[n]}_f)$ is fibrant,
        the zigzag lifts to a weak equivalence of simplicial presheaves 
	\[KO^{[n]}_f(-)\iso \mathbf{Hom}_{sPSh_\bullet(S)}(-\wedge T,KO^{[n+1]}_f).\]
	Taking the adjoint we get an equivalence of motivic spaces 
	\begin{equation}\label{eq:ko_strct}
	T\wedge KO^{[n]}_f(-)\iso KO^{[n+1]}_f(-)
	\end{equation}
	for each $n$. The sequence $(KO^{[n]}_f)_{n\geq 0}$ along with structure maps~ \ref{eq:ko_strct} defines a $T$-spectrum $\mathbf{KO}\in \mathbf{Spt}(S)_T$.  As we have to make choices for fibrant replacements $KO^{[n]}_f$ and the structure maps, the $T$-spectrum $\mathbf{KO}$ is not unique. However, by Theorem~\ref{A:main} we get a unique object in $SH(S)$ up to (not necessarily unique) isomorphism. 
	\begin{definition}\label{def:KO}
		We define $\mathbf{KO}$ to be the naive $T$-spectrum given by the sequence $\mathbf{KO}=(KO^{[n]})_{n\geq 0}$ and the structure maps $T\wedge_{\mathbf{L}} KO^{[n]}\to KO^{[n+1]}$ in $H_\bullet(S)$ induced by the weak equivalences in ~\ref{eq:ko_strct}. By Theorem~\ref{A:main}, $\mathbf{KO}$ defines a unique object in $SH(S)$ up to isomorphism. By abuse of notation we will refer to this object also as $\mathbf{KO}$.
	\end{definition}
	From Theorems \ref{ko:unstable_rep} and \ref{th:thom_iso} we get the following corollary.
	\begin{corollary}\label{cor:rep_main}
		Let $S$ be regular noetherian scheme of finite Krull dimension with $\half\in\Gamma(S,\mc{O}_S)$. For all $X\in Sm_S$ there is an isomorphism of groups
		\[ [\Sigma^\infty X_+, S^{p,q}\wedge\mathbf{KO}]\cong KO^{[q]}_{2q-p}(X)   \]
		where $2q\geq p\geq q\geq 0$ and $[-,-]$ denotes $Hom_{SH(S)}(-,-)$.
	\end{corollary}
	\begin{proof}
		As $S^{p,q}=S^{p-q}\wedge\mathbb{G}^q_m$ and $T\cong S^{2,1}$ in $H_\bullet(S)$ we have a sequence of bijections
		\begin{multline*}
			[\Sigma^\infty X_+, S^{p,q}\wedge\mathbf{KO}]\cong [S^{2q-p}\wedge\Sigma^\infty X_+, T^q\wedge\mathbf{KO}]\\
			\cong [S^{2q-p}\wedge\Sigma^\infty X_+, \mathbf{KO}[-q]]\cong [S^{2q-p}\wedge X_+, KO^{[q]}]_{\mathbb{A}^1}\cong KO^{[q]}_{2q-p}(X)
		\end{multline*}
		for all $2q\geq p\geq q\geq 0$.
	\end{proof}
	The above result suggests a definition of $KO^{[n]}_i(X)$ for $i<0$. 
        Indeed Schlichting in \cite{Sch12} constructs an $\Omega$-spectrum $\mathbb{GW}^{n}(X)$ for any scheme $X$ with an ample family of line bundles which satisfies
	\begin{equation}
	\pi_i \mathbb{GW}^{n}(X)=\begin{cases}
								KO_i^{[n]}(X) &i\geq 0\\
								W^{n-i}(X) &i<0\\
							 \end{cases} 
	\end{equation} 
	for all $i\in\ZZ$ and $n\in\mathbb{N}$. Here $W^{n}(X)$ are Balmer's triangular Witt groups \cite{Bal00}. Setting $KO_i^{[n]}(X)=\pi_i \mathbb{GW}^{n}(X)$ for all $i$ and using \cite[Lem.~5.2]{PW18} we get the following generalisation.
	\begin{theorem}\label{thm:rep_main2}
		Let $S$ be regular noetherian scheme of finite Krull dimension with $\half\in\Gamma(S,\mc{O}_S)$. For all $X\in Sm_S$ there is an isomorphism of groups
		\[ [\Sigma^\infty X_+,~S^{p,q}\wedge\mathbf{KO}]\cong KO^{[q]}_{2q-p}(X)   \]
		where $p\geq q\geq 0$.
	\end{theorem}
	\section{The geometric HP\textsuperscript{1}-spectrum KO\textsuperscript{geo}}\label{sec:geo_spec}
	As before, the stable homotopy category $SH(S)$ is the stabilization of $H_\bullet(S)$ with respect to the functor $(X,x_0)\mapsto T\wedge (X,x_0)$ where $T=\mathbb{A}^1/(\mathbb{A}^1-0)$. There is a Quillen equivalence between the categories $\mathbf{Spt}(S)_T$ and $\mathbf{Spt}(S)_{T^{\wedge 2}}$, of $T$- and $T^{\wedge 2}$-spectra, given by the adjoint pair $(X_n)\mapsto (X_{2n})$ and $(X_n)\mapsto (X_0,T\wedge X_0,X_1,T\wedge X_1,\ldots)$, inducing an isomorphism of stable homotopy categories $SH(S)_T\cong SH(S)_{T^{\wedge 2}}$ \cite[Prop.~2.13]{Jar00}. To construct our desired spectrum we need a different model of the stable homotopy category which utilizes the quaternionic projective space $HP^1=HGr(2,4)$. From this point onwards we will denote the base scheme $S$ by $pt$ to simplify notation.
	\begin{theorem}\label{th:HP1isT2}
		Let $x_0:pt\to HP^1$ be the distinguished point corresponding to the subbundle $[H_-\oplus 0]$. There is an isomorphism  $\eta:(HP^1,x_0)\cong T^{\wedge 2}$ in $H_\bullet(S)$.
	\end{theorem}
	\begin{proof}
			This is \cite[Thm.~9.8]{PW18}. To introduce notation and to illustrate the geometry present when discussing higher Grassmannians, we elaborate some of the arguments below. Consider the open subscheme $\mathbb{A}^4\cong N\into Gr(2,4)$ classifying rank $2$ subbundles $U\mono \mc{O}^4$ whose projection onto $0\oplus 0\oplus\mc{O}^2$ is an isomorphism. We have $N=N^+\oplus N^-$, where $N^+=HP^1\cap Gr(2,0\oplus\mc{O}\oplus\mc{O}^2)$ and $N^-=HP^1\cap Gr(2,\mc{O}\oplus 0\oplus\mc{O}^2)$ are closed subschemes of $HP^1$ which are isomorphic to $\mathbb{A}^2$. By \cite[Thm.~3.4]{PW10}, $HP^1- N^+$ is the quotient of $\mathbb{A}^5=\mathbb{A}^2\times\mathbb{A}^2\times\mathbb{A}^1$ by the free $\mathbb{G}_a$-action,
		\[ t\cdot(a,b,r)=(a,b+ta,r+t(1-a \begin{pmatrix}
		0 &-1\\
		1 &0
		\end{pmatrix}b)).\]
		The inclusion $\mathbb{A}^1\into \mathbb{A}^5$ given by $t\mapsto (0,0,0,0,t) $ is then $\mathbb{G}_a$-equivariant and an $\mathbb{A}^1$-equivalence. Therefore we have an induced $\mathbb{A}^1$-equivalence of the quotients $x_0:pt\to HP^1-N^+$ given by the subspace $0\oplus 0\oplus H_-$. The commutative square,  
		\[		
		\begin{tikzcd}
		pt \arrow[d,"\sim"] \arrow[r] & HP^1 \arrow[d,"="]\\
		HP^1- N^+ \arrow[r, hook] & HP^1
		\end{tikzcd}
		\] gives us an equivalence of pointed spaces $(HP^1,x_0)\iso HP^1/(HP^1-N^+)$.
		We have a similar square induced by the $\mathbb{A}^1$-equivalence $N^-\iso N$,
		\[		
		\begin{tikzcd}
		N^-- 0 \arrow[d,"\sim"] \arrow[r] & N^- \arrow[d,"\sim"]\\
		N- N^+ \arrow[r, hook] & N
		\end{tikzcd}
		\]
		where the left hand side is an equivalence as $N-N^+=\mathbb{A}^4-\mathbb{A}^2$ is a rank $2$ vector bundle over $N^--0$. Hence we have $N^-/(N^--0)\cong N/(N-N^+)$ in $H_\bullet(S)$ and there is a zigzag of $\mathbb{A}^1$-equivalences
		\[		
		\begin{tikzcd}
		T^2\cong \mathbb{A}^2/(\mathbb{A}^2-0)\cong N^-/N^-- 0 \arrow[d,"\sim"] \arrow[r]\arrow[rd] & HP^1/(HP^1-N^+) & (HP^1,x_0)\arrow[l,"\sim"]\\
		N/N- N^+  & N\cap HP^1/((N\cap HP^1)-N^+)\arrow[l,"\mathrm{excision}"]\arrow[u,"\mathrm{excision}"]
		\end{tikzcd}
		\]
		Using the $2$ out of $3$ property twice this gives us an $\mathbb{A}^1$-equivalence $T^{\wedge 2}\iso HP^1/(HP^1-N^+)$ and hence $T^{\wedge 2}\cong (HP^1,x_0)$ in $H_\bullet (S)$.  
	\end{proof}
	Theorem~\ref{th:HP1isT2} provides equivalences of stable homotopy categories $SH(S)\cong SH(S)_{T^{\wedge2}}\cong SH(S)_{HP^1}$. We now have the desired model of $SH(S)$. We will construct $\mathbf{KO}^{geo}$ as a naive $HP^1$-spectrum in the sense of appendix \ref{sec:naive-spectra}. For any $n\in\ZZ$, we have a canonical decomposition $V_{n}\perp V^\perp_{n}\cong H^{n}_+$ over $RGr(n,2n)$ which corresponds to the identity map $RGr(n,2n)\to RGr(n,2n)$. Similarly we have the canonical decompostion $U_{2n}\perp U^\perp_{2n}\cong H^{2n}_-$ over $HGr(2n,4n)$.
	Let $\ZZ\times HGr:=\colim_n [-n,n]\times HGr(2n,4n)$, where $[-n,n]\times X$ denotes the  disjoint union of $2n+1$ copies of $X$. The morphism $HGr(2n,4n)\to HGr(2n+2,4n+4)$  is given by the subbundle
	\[ U_{2n}\perp H_-\mono (H^n_-\perp H_-)\perp(H^n_-\perp 0)\mono (H^n_-\perp H_-)\perp (H^n_-\perp H_-) \] 
	where $HGr(2n,4n)$ is identified with $HGr(2n,H^n_-\perp H^n_-)$. Defining $[-n,n]'$ by
	\[ [-n,n]'=\{i\in\ZZ|-n\leq i\leq n\: and\: i\equiv n\:\mathrm{mod}\:2\}\] 
	the infinite real Grassmannian,
	\[\ZZ\times RGr:=\colim_n [-2n,2n]'\times RGr(2n,4n)\cup [-2n+1,2n-1]'\times RGr(2n-1,4n-2)\]
	is defined along similar lines. The distinguished points $h_0:pt\to\ZZ\times HGr$ and $r_0:pt\to\ZZ\times RGr$ correspond to the subbundles $H^n_-\perp 0\mono H^n_-\perp H^n_-$ and $H^n_+\perp 0\mono H^n_+\perp H^n_+$ over $\{0\}\times HGr(2n,4n)\subset [-n,n]\times HGr(2n,4n)$ and $\{0\}\times RGr(2n,4n)\subset [-2n,2n]'\times RGr(2n,4n)$ respectively for each $n$. 
	\begin{lemma}\label{lem:structure-hgr-rgr}
		For all $n\geq 0$, there exist morphisms of pointed smooth schemes
		\[f_{2n}:([-n,n]\times HGr(2n,4n))\times HP^1\to RGr(16n,32n) \]
		such that the following hold. 
		\begin{enumerate}
			\item The restriction $f_{2n|(0,H^{n}_-\perp 0)\times HP^1}$ is given by the subbundle $H^{8n}_+\perp 0$.
			\item The restriction $f_{2n|[-n,n]\times HGr(2n,4n)\times (H_-\perp 0)}$ is given by an embedding $H^{8n}_+\mono H^{16n}_+$ which is $\mathbb{A}^1$-homotopic to the embedding given by $H^{8n}_+\perp 0$.
			\item These morphisms and $\mathbb{A}^1$-homotopies are compatible with inclusions of schemes $HGr(2n,4n)\to HGr(2(n+1),4(n+1))$ and $RGr(16n,32n)\to RGr(16(n+1),32(n+1))$. 
		\end{enumerate}
	\end{lemma}
	\begin{proof}
		For  simplicity, given two vector bundles $U$ and $V$ over two distinct schemes $X$ and $Y$, we denote by $U\boxtimes V$ the vector bundle $p_1^*U\otimes p_2^*V$ over $X\times Y$ where $p_i$ are the projections. We have decompositions $U_{2n}\perp U^\perp_{2n}\cong H^{2n}_-$ for all $n$. Remember that the tensor product of two symplectic spaces is a symmetric space and in particular $H^{n}_-\boxtimes H_-^{ m}\cong H_+^{2mn}$. For each $n\in\mathbb{N}$ and $i\in[-n,n]$, we have inclusions of symmetric spaces over $HGr(2n,4n)\times HP^1$, 
		\begin{equation}\label{bo_structure_map:eq1}
		\begin{aligned}
			U_{2n}\boxtimes U_2\mono H^{2n}_-\boxtimes U_2,\\
			H_-^{n-i}\boxtimes U^{\perp}_2\mono H_-^{2n}\boxtimes U^{\perp}_2,\\
			U^{\perp}_{2n}\boxtimes H_-\mono H^{2n}_-\boxtimes H_- ,\\
		 	H_-^{n+i}\boxtimes H_-\mono H_-^{2n}\boxtimes H_-. 
		 \end{aligned}
		 \end{equation}
		Putting these together, we get 
		\begin{equation}
		(U_{2n}\boxtimes U_2)\perp(H^{n-i}_-\boxtimes U^\perp_2)\perp(U^{\perp}_{2n}\boxtimes H_-)\perp  (H_-^{n+i}\boxtimes H_-)  
		\end{equation}
		which is a rank $16n$ symmetric subspace of the rank $32n$ symmetric space
		\begin{equation}\label{bo_structure_map:eq2}
		(H^{2n}_-\boxtimes U_2) \perp (H^{2n}_-\boxtimes U^\perp_2)\perp (H^{2n}_-\boxtimes H_-)\perp  (H_-^{2n}\boxtimes H_-).  
		\end{equation}
		This space is isometric to $H^{16n}_+$ and a choice of isometry gives us $2n+1$ subspaces of $H^{ 16n}_+$ (one for each $i$), each of which is rank $16n$. Hence we have a morphism
		\begin{equation}
		f_{2n}:([-n,n]\times HGr(2n,4n))\times HP^1\to RGr(16n,32n).
		\end{equation}
		The composition  
		\[HGr(2n,4n)\times HP^1\to HGr(2n+2,4n+4)\times HP^1 \to RGr(16n+16,32n+32),\]
		for a fixed $i$ then is given by the subspace
		\[ ((U_{2n}\perp H_-)\boxtimes U_2)\perp ((U^{\perp}_{2n}\perp H_-)\boxtimes H_-)\perp ((H^{n-i}_-\perp H_-)\boxtimes U^\perp_2)\perp H^{2n+2+2i}_+.\]
		To make this equal to the composition
		\[HGr(2n,4n)\times HP^1\to RGr(16n,32n)\to RGr(16n+16,32n+32)\]
		we need to choose the isometry used to define $f_{2n}$ carefully. First,  we  index $H^{2n}_-$ as $\perp^{2n}_{i=1} {H^{(i)}_-}$. 
		We then get isometries
		\begin{equation}\label{eq:embedHplus}
		(H^{(i)}_-\boxtimes U_2) \perp ({H^{(i)}_-}\boxtimes U^\perp_2)\perp (H^{(i)}_-\boxtimes H_-)\perp (H^{(i)}_-\boxtimes H_-)\cong H^{8}_+
		\end{equation}	for each $i\in \{1,\ldots,2n\}$. Putting these together we get, 
		\begin{equation}
		(H^{2n}_-\boxtimes U_2) \perp (H^n_-\boxtimes U^\perp_2)\perp (H^n_-\boxtimes H_-)\perp (H^n_-\boxtimes H_-)\cong H^{16n}_+ 
		\end{equation}
		which is the desired isometry. Pulling back this isometry along the inclusion $(0,H^n_-\perp 0)\times HP^1\into [-n,n]\times HGr(2n,4n))\times HP^1$ turns the symplectic subspace~\ref{bo_structure_map:eq2} into
		$H^{8n}_+\perp 0$ and pulling back the isometry along $[-n,n]\times HGr(2n,4n)\times (H_-\perp 0)$ turns it into 
		\[(U_{2n}\boxtimes(H_-\perp 0))\perp (H^{n-i}_-\boxtimes(0\perp H_-))\perp (U^\perp_{2n}\boxtimes(H_-\perp 0))\perp(H^{n+i}_-\boxtimes(0\perp H_-)) \] 
		which is a subbundle isometric to $H^{8n}_+$. Permuting the summands of $H^{16n}_+$ we can send the above subbundle to $H^{8n}_+\perp 0\mono H^{8n}_+\perp H^{8n}_+$. As given in the proof of Theorem~\ref{iso_ksp} these permutations are $\mathbb{A}^1$-homotopic to identity. 
	\end{proof}
	 The next lemma can be proved analogously,
	\begin{lemma}\label{lem:structure-rgr-hgr}
		For all $n\geq 0$, there exist morphisms of pointed smooth schemes
		\[g_{n}:([-n,n]'\times RGr(n,2n))\times HP^1\to HGr(8n,16n) \]
		such that the following holds:
		\begin{enumerate}
			\item The restriction $g_{n|(0,H^{2n}_+\perp 0)\times HP^1}$ is given by the subbundle $H^{8n}_-\perp 0$ when $n$ is even.
			\item The restriction $g_{n|[-n,n]\times RGr(2n,4n)\times (H_+\perp 0)}$ is given by an embedding $H^{8n}_-\mono H^{16n}_-$ which is $\mathbb{A}^1$-homotopic to the embedding given by $H^{8n}_-\perp 0$.
			\item These morphisms and $\mathbb{A}^1$-homotopies are compatible with inclusions of schemes $RGr(n,2n)\to RGr(n+1,2(n+1))$ and $HGr(8n,16n)\to HGr(8(n+1),16(n+1))$.
		\end{enumerate}
	\end{lemma}
	\begin{proof}
		The proof is the same as for Lemma~\ref{lem:structure-hgr-rgr}, except for a change in indices. The morphism $g_n$ is defined on the $i^{th}$ component, where $i\in[-n,n]'$, by the subspace
		\begin{equation}
		(V_n\boxtimes V_1)\perp(H^{\frac{n-i}{2}}_+\boxtimes V^\perp_1)\perp (V^\perp_n\boxtimes H_+) \perp (H^{n+1})
		\end{equation}
	\end{proof}
	Using the morphisms of schemes $f_n$ and $g_n$ we can construct a naive $HP^1$-spectrum. 
	\begin{theorem}
          For any scheme $S$, there exists a naive $HP^1$-spectrum $\mathbf{KO}^{geo}_S$ 
          such that for all $i\geq 0$, $\mathbf{KO}^{geo}_{2i+1}=\ZZ\times HGr$ and  $\mathbf{KO}^{geo}_{2i}=\ZZ\times RGr$ as motivic spaces and $\Omega^2_{HP^1}\mathbf{KO}^{geo}\cong \mathbf{KO}^{geo}$ in $SH(S)_{HP^1}$.
	\end{theorem}
		Here we are abusing notation and using $\mathbf{KO}_S^{geo}$ to mean both the naive $HP^1$-spectrum and the associated object in $SH(S)$.	
        \begin{proof}
          From Lemma~\ref{lem:structure-hgr-rgr} and Lemma~\ref{lem:structure-rgr-hgr} we get morphisms of ind-schemes, 
          \begin{equation}\label{eq:struct_ko}
          f:HP^1\times\ZZ\times HGr\to RGr\quad g:HP^1\times\ZZ\times RGr\to HGr
          \end{equation}
          such that $f_{|HP^1\vee HGr}$ and $g_{|HP^1\vee RGr}$ are $\mathbb{A}^1$-homotopic to the constant zero morphism. These induce structure maps $f:HP^1\wedge_{\mathbf{L}}(\ZZ\times HGr)\to \ZZ\times RGr$ and $g:HP^1\wedge_{\mathbf{L}}(\ZZ\times RGr)\to \ZZ\times HGr$ in $H_\bullet(S)$, where $\wedge_{\mathbf{L}}$ is the left derived functor of the smash product.
          Periodicity follows by construction.
        \end{proof}
	The periodicity above implies $(8,4)$-periodicty in the standard bigrading of motivic spectra.  
	\begin{remark}
		Our construction of $\mathbf{KO}^{geo}$ is along the same lines as the construction of $\mathbf{BO}$ in \cite{PW18}. However we have shown that the condition of $2$ being invertible is not needed. We have also used the theory of naive spectra elaborated in Appendix \ref{sec:naive-spectra} which simplifies that construction of the structure maps of $\mathbf{KO}^{geo}$ and shows that  only the class of the structure maps in $H_\bullet(S)$ matter. 
	\end{remark}  
	\section{Properties of KO\textsuperscript{geo}} \label{sec:prop}
	Having constructed $\mathbf{KO}^{geo}$, we will look at some of its properties. The first notable property is that it is absolute in the following sense.

	\begin{theorem}\label{BO_pullback}
		For any morphism $f:S \to T$ of schemes, there is a canonical isomorphism $Lf^*\mathbf{KO}^{geo}_{T}\iso \mathbf{KO}^{geo}_{S}$ in $SH(S)$. In particular, for any scheme $S$, $\mathbf{KO}^{geo}_S$ is isomorphic to the pullback of $\mathbf{KO}^{geo}_{\ZZ}=\mathbf{KO}^{geo}_{\Spec(\ZZ)}$ by the structure map $S\to \Spec(\ZZ)$.
	\end{theorem}
        \begin{proof}
          Let $f:S\to T$ be a morphism of schemes. Pullback induces canonical identifications $f^*(HGr(2r,2n)_{T})\cong HGr(2r,2n)_{S}$ and $f^*(\ZZ\times HGr_{T})\cong \ZZ\times HGr_{S}$. This pullback isomorphism is also compatible with the structure maps of $\mathbf{KO}^{geo}$, as the tensor product of forms is preserved under pullback. Using the closed model structure given in \cite{PPR09} $f^*$ is a left Quillen functor completing the proof. 
        \end{proof}
        The next interesting property is that $\mathbf{KO}^{geo}$ over $\Spec(\ZZ[\half])$ gives us back the motivic spectrum $\mathbf{BO}$ constructed in \cite{PW18}. This implies a corresponding representability result for hermitian K-theory. To state this result we will need a model of $\mathbf{BO}$ as a naive $HP^1$-spectrum. 
        For any rank $2$ symplectic bundle $(E,\phi)$ over a scheme $X$, the structure map $\phi:E\otimes E\to \mc{O}_X$ induces an isomorphism $\Lambda^2E\iso\mc{O}_X$. The canonical morphism $U_{2}\mono {H^{ 2}_-}$ over $HP^1$ restricts to a set of four maps $U_{2}\mono \mc{O}_{HP^1}$. The pair which factors through $U_{2}\mono H_-$ differ up to isomorphism only by a sign. Therefore denote these pairs by $(x_0,-x_0)$ and $(x_\infty,-x_\infty)$ respectively (this notation in consistent with the fact the these are isomorphisms when pulled back along points $x_0$ and $x_\infty$). Consider the symmetric form  
        \[
        \begin{tikzcd}
        \mc{O}_{HP^1}\arrow[r]\arrow[d,"-1"] & U_{2}\arrow[r,"x_0"]\arrow[d,"\phi_{2,4}"] &\mc{O}_{HP^1}\arrow[d,"1"]\\
        \mc{O}_{HP^1}\arrow[r,"(-x_\infty)^\vee"] & U^\vee_{2}\arrow[r] &\mc{O}_{HP^1}
        \end{tikzcd}
        \]
        in $Ch^b(Vect(HP^1\times X),-\eta,q)$, indexed from degrees $0$ to $2$. By construction this form is equal to $[U_{2}]-[H_-]$ in $KSp(HP^1\times X)$ under \[KO^{[2]}(HP^1\times X)\iso KSp^{[0]}(HP^1\times X)\iso KSp(HP^1\times X)\] 
        and is the pullback of $\kappa(U_{2},\phi)$ along the zero section $z:HP^1\to U_{2}$. We will call this element of $KO_0^{[2]}(HP^1\times X)$ the \emph{Borel class} $-b_1(U_{2})$. The Borel class will give us the desired structure map. Let us denote by $\mathbf{BO}_{HP^1}$ the image of $\mathbf{BO}$ in $SH(S)_{HP^1}$.
        \begin{theorem}\label{ko_hp1}
        	Let $S$ be a regular Noetherian scheme of finite Krull dimension with $\half\in\Gamma(\mc{O}_S,S)$. The structure morphisms of $\mathbf{BO}_{HP^1}$ are represented by maps 
        	\[KO^{[n]}(-)\to KO^{[n+2]}(-\times HP^1),\quad C.\mapsto C.\boxtimes (-b_1(\mathrm{U}_{2}))\]
        	for all $n$.
        \end{theorem}
        \begin{proof}
        	It is enough to show that the image of the Borel class under the zigzag of weak equivalences between $HP^1$ and $T^{\wedge 2}$ is $th(\mc{O},id)\boxtimes th(\mc{O},id)$. Firstly the sequence $\mc{O}_{HP^1}\to U_{2}\xrightarrow{x_0}\mc{O}_{HP^1}$ is exact when restricted to $HP^1- N^+$. Hence $-b_1(U_{2})$ is an element of $KO_0^{[n+2]}(HP^1,HP^1- N^+)$. The pullback of the morphism of vector bundles $U_{2}\xrightarrow{x_0}\mc{O}_{HP^1}$ along $\mathbb{A}^2\cong N^-\to HP^1$ is $\mc{O}_{\mathbb{A}^2}^{\oplus 2}\xrightarrow{(t_0,t_1)} \mc{O}_{\mathbb{A}^2}$, where $t_0$ and $t_1$ are the variables in $\mathbb{A}_{\Spec(R)}^2\cong \Spec(R[t_0,t_1])$. Therefore $-b_1(U_{2})$ becomes
        	\[
        	\begin{tikzcd}
        	\mc{O}_{\mathbb{A}^2}\arrow[r]\arrow[d,"-1"] &\mc{O}^{\oplus 2}_{\mathbb{A}^2} \arrow[d]\arrow[r,"(t_0{,} t_1)"] &\mc{O}_{\mathbb{A}^2} \arrow[d,"1"]\\
        	\mc{O}_{\mathbb{A}^2} \arrow[r,"(-t_1{,} -t_0)"] & \mc{O}^{\oplus 2}_{\mathbb{A}^2} \arrow[r] &\mc{O}_{\mathbb{A}^2}
        	\end{tikzcd}
        	\]
        	which is $th(\mc{O},id)\boxtimes th(\mc{O},id)$ as required.
        \end{proof} 
    	We can now state the desired result.
	\begin{theorem}
		There exists an isomorphism $\mathbf{KO}^{geo}\cong \mathbf{BO}_{HP^1}$ as objects in $SH(\Spec(\ZZ[\half]))$. Hence $\mathbf{KO}^{geo}$ represents hermitian K-theory over regular Noetherian schemes of finite Krull dimension with $\half\in\Gamma(S,\mc{O}_S)$.
	\end{theorem}
	\begin{proof}
		The isomorphisms $\tau:\ZZ\times HGr\iso KSp$ in $H_\bullet(S)$ follows from Theorem~\ref{iso_ksp}. To show that the diagram
		\[
		\begin{tikzcd}
		HP^1\wedge_{\mathbf{L}} HP^1\wedge (\ZZ\times HGr)\arrow[r]\arrow[d,"1\wedge1\wedge\tau_{4n-2}"] &\ZZ\times HGr\arrow[d,"\tau{4n+2}"]\\
		HP^1\wedge_{\mathbf{L}} HP^1\wedge_{\mathbf{L}} KO^{[4n-2]}\arrow[r] & KO^{[4n+2]}
		\end{tikzcd}
		\]
		commutes in $H_\bullet(S)$, we use the fact that $\ZZ\times HGr$ is an ind-scheme and hence the restriction $\{i\}\times HGr(2r,2n)\into \ZZ\times HGr\iso KSp$ is classified by an element in $KSp_0(HGr(2r,2n))$. The map $BSp_{2n}\to Bi\mathbf{Symp}$ in $H_\bullet(S)$ is given at the level of schemes by sending principal $Sp_{2n}$-bundles to the associated symplectic bundle. Consequently the map $HGr(2r,2n)\to BSp_{2n}\to Bi\mathbf{Symp}$ corresponds to the tautological symplectic bundle $U_{2r,2n}\to HGr(2r,2n)$. From this and the definition of the map $\ZZ\times  M_\infty\to R\Omega^1BM$, it follows that the isomorphism $\tau:\ZZ\times HGr\to KSp$ satisfies 
		\begin{multline*}
		\tau_{|\{i\}\times HGr(2r,2n)}=[U_{2r,2n}]+(i-r)[H_-] \\\in KSp_0(HGr(2r,2n)) \to Hom_{H_\bullet(S)}(HGr(2r,2n),KSp).
		\end{multline*}
		The last map is an isomorphism over regular Noetherian schemes $S$ with $\half\in\Gamma(S,\mc{O}_S)$. We denote by $\tau_{4k+2}$ the composition 
		\[\ZZ\times HGr\xrightarrow{\tau} KSp\iso KO^{[4k+2]}\]
		where $KSp\iso KO^{[4k+2]}$ is the isomorphism \ref{spo_iso} for $n=0$.  We then have 
		\[
			{\tau_{4k+2}}_{|\{i\}\times HGr(2n,4n)}=([U_{2n,4n}]+(i-2n[H_-]))[2k+1]
		\] 
		and similarly the map $HP^1\wedge HP^1\wedge(\ZZ\times HGr)\to \ZZ\times HGr\to KSp$ restricted to $HP^1\times HP^1\times\{i\}\times HGr(2n,4n)$ is given by \[([U_{2}]-[H_-])\boxtimes([U_{2}]-[H_-])\boxtimes([U_{2n,4n}]+(i-2n)[H_-]).\] To see this note that $[U^\perp_{2n,4n}]=2n[H]-[U_{2n,4n}]$. But tensoring twice with $([U_{2}]-[H_-])$ is exactly the structure map of $\mathbf{BO}_{HP^1}$ (\ref{ko_hp1}) and hence the diagram commutes when restricted to the finite Grassmannians. As $\ZZ\times HGr$ is the colimit of $\{i\}\times HGr(2n,4n)$ we have a map
		\[Hom_{H_\bullet(S)}(\colim_n HGr(2n,4n),X)\to \lim_n  Hom_{H_\bullet(S)}(HGr(2n,4n),X)\]
		for any $X\in\mathbf{Spc}_\bullet(S)$. This is an isomorphism if \[Hom_{H_\bullet(S)}(S_s^1\wedge HGr(2n+2,4n+4),X)\to Hom_{H_\bullet(S)}(S_s^1\wedge HGr(2n,4n),X)\]
		is a surjection. To see this, take a fibrant replacement of $X$ to get the set of simplicial homotopy classes. Surjectivity then implies that we can lift a collection of homotopy classes, uniquely up to homotopy, to the colimit. This holds for $X=KO^{[k]}$ as then we have 
		\[Hom_{H_\bullet(S)}(S_s^1\wedge HGr(2n,4n),KO^{[k]})\cong KO^{[k]}_1(HGr(2n,4n))\] 
		and the maps $KO^{[k]}_i(HGr(2n+2,4n+4))\to KO^{[k]}_i(HGr(2n,4n))$ are surjections by \cite[Thm.~11.4]{PW10} applied to $KO_*^{[*]}$ which is a cohomology theory with a $-1$-commutative ring structure \cite[Thm.~1.4]{PW18}.
	\end{proof}
	\begin{remark}
		Note that we needed $\half\in\Gamma(S,\mc{O}_S)$ to get the isomorphism $KSp\iso KO^{[2k+1]}$ as only then do we have the identification between skew-symmetric and alternating forms. 
	\end{remark}
	We can also extend the cellularity result in \cite{RSO16} to arbitrary schemes. We will use the definition of cellular spectra from \cite{DI05}.
	\begin{definition}\label{cell_defn}
		For any scheme $S$, let $\mathbf{Spt}_{cell}(S)$ be the smallest full subcategory of $\mathbf{Spt}(S)$ satisfying the following:
		\begin{enumerate}
			\item One has $S^{p,q}\in\mathbf{Spt}_{cell}(S)$ for all $p,q\in\ZZ$.
			\item If $\mathbf{F}$ is stably equivalent to $\mathbf{E}$ for some $\mathbf{E}\in\mathbf{Spt}_{cell}(S)$ then $\mathbf{F}\in \mathbf{Spt}_{cell}(S)$.
			\item For any diagram $D\to \mathbf{Spt}_{cell}(S)$, $\hocolim D$ is in $\mathbf{Spt}_{cell}(S)$. 
		\end{enumerate}
		As $\mathbf{Spt}_{cell}(S)$ is closed under stable equivalences, it defines a subcategory $SH(S)_{cell}$ of $SH(S)$. We call elements of $\mathbf{Spt}_{cell}(S)$ \emph{cellular spectra}. 
		Given a morphism $f:S_1\to S_2$, we have the following.
	\end{definition}
		\begin{lemma}\label{cell_pullback}
		For any morphism of schemes $f:S_1\to S_2$, $Lf^*:SH(S_2)\to SH(S_1)$ restricts to a morphism of cellular objects $Lf^*:SH(S_2)_{cell}\to SH(S_1)_{cell}$.
	\end{lemma}
	\begin{proof}
		This follows from the fact that $Lf^*$ preserves all motivic spheres and homotopy colimits.
	\end{proof}
	We wish to prove the following.
	\begin{theorem}\label{th:cell}
		Let $S$ be any scheme. The motivic spectrum $\mathbf{KO}^{geo}_S$ is cellular.
	\end{theorem}
		To prove this we will first show that the suspension spectra $\Sigma^\infty HG(2r,2n)_+$ are cellular. They constitute the case $m=0$ of the following statement.
	\begin{lemma}
		Let $m\geq 0$. The suspension spectrum of the Thom space of $U^{\oplus m}_{2r,2n}$ on $HGr(2r,2n)$ is a finite cellular spectrum. 
	\end{lemma}
	\begin{proof}
		By Lemma~\ref{cell_pullback} it is enough to prove this over $\Spec(\ZZ)$. The proof is by induction on $r$ and $n$. As $HGr(0,2n)\cong pt$, the statement holds for $r=0$. Extending the definitions in Theorem~\ref{th:HP1isT2} we define $N^+=HGr(2r,2n)\cap Gr(2r,0\oplus\mc{O}^{2r}\oplus\mc{O}^{2n})$ and $N^-=HGr(2r,2n)\cap Gr(2r,\mc{O}^{2r}\oplus 0\oplus\mc{O}^{2n})$ as closed subschemes of $HGr(2r,2n)$.  The direct sum of these bundles $N=N^+\oplus N^-$ is the normal bundle of the embedding $HGr(2r,2n-2)\to HGr(2r,2n)$ \cite[Thm~4.1]{PW10}. Let $Y$ be the open subscheme $HGr(2r,2n)\setminus N^+$, then by \cite[Lem.3.5]{Spi10} the cofiber of the map 
		\[Th({U^{\oplus m}_{2r,2n}}_{|Y})\to Th(U^{\oplus m}_{2r,2n}),\] 
		is isomorphic to $Th({U_{2r,2n}^{\oplus m}}_{|N^+}\oplus\mc{N})$ where $\mc{N}$ is the normal bundle of the closed embedding $N^+\to HGr(2r,2n)$. We have ${U_{2r,2n}}_{|N^+} \cong \pi^*_+ U_{2r,2n-2}$ by \cite{PW10} and in fact the proof shows us that $\pi^*_+ U_{2r,2n-2}\cong \mc{N}$. We therefore have a cofiber sequence
		\[Th({U^{\oplus m}_{2r,2n}}_{|Y})\to Th(U^{\oplus m}_{2r,2n})\to Th(\pi^*_+ U^{\oplus m+1}_{2r,2n-2}) .\]
		where $\pi_+$ is the structure map of a vector bundle. Therefore, the induced morphism $Th(\pi^*_+ U^{\oplus m+1}_{2r,2n-2})\to Th(U^{\oplus m+1}_{2r,2n-2})$  is an unstable weak equivalence. 
                By induction on $n$ we have reduced to showing that $\Sigma^\infty Th({U^{\oplus m}_{2r,2n}}_{|Y})$ is cellular. By \cite[Thm.~5.1]{PW10} we have a zig-zag
		\[Y\leftarrow Y_1 \leftarrow Y_2\to HG(2r-2,2n-2) \]
		where every map is an affine bundle, such that moreover there is an isomorphism of symplectic bundles  
		\[{U_{2r,2n}}_{|Y_2}\cong \mc{O}^{\oplus 2}_{Y_2}\oplus {U_{2r-2,2n}}_{|Y_2} \]
		and the map $Y_2\to HG(2r-2,2n-2)$ has a section by the proof of \cite[Thm.~5.2]{PW10}, whence every scheme in the sequence has a point. By  Theorem~\ref{thm:thom-iso} we then have equivalences 
		\[Th({U_{2r,2n}}_{|Y} )\simeq Th({U_{2r,2n}}_{|Y_2})\cong Th(\mc{O}_{Y_2}^{2m}\oplus {U^{\oplus m}_{2r-2,2n}}_{|Y_2})\simeq S^{4m,2m}\wedge Th(U^{\oplus m}_{2r-2,2n}) .\]
		Induction completes the proof.	
	\end{proof}
	This is enough to prove the required result.
	\begin{proof}[Proof of Theorem~\ref{th:cell}]
		The proof is essentially given in \cite{RSO16}. As before we have $\mathbf{KO}^{geo}=\hocolim_n \Sigma^{-4n,-2n}\Sigma^{\infty}\ZZ\times HGr$. Therefore by \cite[Lemma 3.4]{DI05} and Definition~\ref{cell_defn} (3) it is enough to show that $\Sigma^{\infty} HGr(2n,4n)_+$ is cellular for each $n$.
	\end{proof}
	\begin{appendices}
		\section{Naive spectra}\label{sec:naive-spectra}
	Throughout this section $\mc{C}$ is a pointed cofibrantly generated model category with fibrant and cofibrant replacement functors $R$ and $Q$ respectively. We denote by $H(\mc{C})$ the corresponding homotopy category. Given a Quillen adjunction
	\[(T,U,\eta):\mc{C}\leftrightarrows \mc{C}, \] 
	the category of $T$-spectra $Sp^\mathbb{N}(\mc{C},T)$ has as objects sequences $\{E_i\}_{i\geq 0}$ of objects in $\mc{C}$ along with assembly maps 
	\[e_i:T(E_i)\to E_{i+1}. \]  	
	The category $Sp^{\mathbb{N}}(\mc{C},T)$ inherits several model structures from $\mc{C}$. Here we will consider the levelwise and stable projective model structures given in \cite{Hov01}. We denote the associated homotopy categories by $H^l(Sp^{\mathbb{N}}(\mc{C},T))$ and $SH(\mc{C},T)$ respectively. 
	\begin{definition}[Naive spectra]\label{A_HS_def}
		A \emph{naive T-spectrum} $(E_\cdot,e_\cdot)$ is a sequence $\{E_n\}_{n\in \mathbb{N}}$ of objects in $H(\mc{C})$ equipped with assembly morphisms 
		\[e_n\in Hom_{H(\mc{C})}(\mathbf{L}T(E_n),E_{n+1}) .\]
		A morphism of naive $T$-spectra $\phi:(E_\cdot,e_\cdot)\to (F_\cdot,f_\cdot)$ is a collection of morphisms $\phi_n:E_n\to F_n$ in $H(\mc{C})$ such that the relevant diagrams commute in $H(\mc{C})$. We denote this category by $Sp^\mathrm{naive}(\mc{C},T)$.
	\end{definition}
	Any $T$-spectrum $\mathbf{E}$ defines canonically a naive $T$-spectrum $\naive(\mathbf{E})$ with underlying sequence $\mathbf{E}_n$ and assembly morphisms the image of
	\[\mathbf{L}T(\mathbf{E}_n)=T(Q\mathbf{E}_n)\to T(\mathbf{E}_n)\xrightarrow{e_n} \mathbf{E}_{n+1}\]
	in $H(\mc{C})$ for every $n$. This gives us a functor $\naive:Sp(\mc{C},T)\to Sp^{\mathrm{naive}}(\mc{C},T)$.
	\begin{remark}
		The notion of naive spectra presented here is a generalization of the one given in \cite{Riou06}.
	\end{remark}
	\begin{lemma}\label{l:surj}
		The functor $\naive:Sp^{\mathbb{N}}(\mc{C},T)\to Sp^{\mathrm{naive}}(\mc{C},T)$ is essentially surjective.
	\end{lemma}
	\begin{proof}
		Given any naive $T$-spectrum $(E_n,e_n)_{n\in \mathbb{N}}$, choose bifibrant models $E'_n\in\mc{C}$ of $E_n$. We then have isomorphisms
		\[ Hom_{H(\mc{C})}(\mathbf{L}T(E_n),E_{n+1})\cong  Hom_{\mc{C}}(T(E'_n),E'_{n+1})/\simeq		 \]
		allowing us to choose a lift $e'_n:T(E'_n)\to E'_{n+1}$ of $e_n$ for every $n$. We then have a $T$-spectrum $\mathbf{E}'$ whose image under $\naive$ is isomorphic to $(E_n,e_n)_{n\in \mathbb{N}}$.	
	\end{proof}
	\begin{lemma}\label{l:iso}
		Let $\mathbf{E}=(E_0,E_1\ldots),\mathbf{F}=(F_0,F_1,\ldots)\in Sp^{\mathbb{N}}(\mc{C},T)$ with assembly maps $e_n:T(E_n)\to E_{n+1}$ and $f_n:T(F_n)\to F_{n+1}$ respectively. If there is an isomorphism $\phi:\naive(\mathbf{E})\iso \naive(\mathbf{F})$ in $Sp^{\mathrm{naive}}(\mc{C},T)$, then $\mathbf{E}$ and $\mathbf{F}$ are isomorphic in $H^l(Sp^{\mathbb{N}}(\mc{C},T))$.
	\end{lemma}
	\begin{proof}
		We have isomorphisms $\phi_n:E_n\iso F_n$ in $H(\mc{C})$ such that the diagrams
		\[
		\begin{tikzcd}
		\mathbf{L}T(E_n)\arrow[r]\arrow[d,"\mathbf{L}T(\phi_n)"] &E_{n+1}\arrow[d,"\phi_{n+1}"]\\
		\mathbf{L}T(F_n)\arrow[r] &F_{n+1}
		\end{tikzcd}
		\] 
		commute for all $n$. As $\mathcal{C}$ is cofibrantly generated, so is $H^l(Sp^\mathbb{N}(\mathcal{C},T))$ and hence it has a fibrant and cofibrant replacement functor \cite{Hov01}. Using the cofibrant-fibrant replacement of the spectra $\mathbf{E}$ and $\mathbf{F}$ we can reduce to case when the spectra are levelwise cofibrant-fibrant. In this case each $\phi_n$ lifts to a weak equivalence in $\mc{C}$ and commutativity of the diagrams implies that the two maps $\phi_{n+1}e_n$ and $f_nT(\phi_n)$ are homotopic \cite[Prop. 1.2.5]{Ho07}. For each $n\in\mathbb{N}$, let $E_n\times I$ be the functorial cylinder object for $E_n$. As $T(E_n)$ is cofibrant and $F_{n+1}$ is fibrant, $T(E_n\times I)$ is a cylinder object for $T(E_n)$ and there is a left homotopy
		\[
		\begin{tikzcd}
		T(E_n)\arrow[r,"e_n"]\arrow[d,"i_0"] &E_{n+1}\arrow[d,"\phi_{n+1}"]\\
		T(E_n\times I)\arrow[r,"H_n"] &F_{n+1}\\
		T(E_n)\arrow[r,"T(\phi_n)"]\arrow[u,"i_1"] &T(F_{n})\arrow[u,"f_n"],
		\end{tikzcd}		
		\]
		between $\phi_{n+1}e_n$ and $f_nT(\phi_n)$. The mapping cylinder $M_{\phi_n}$ is the pushout of the diagram 
		\[
		\begin{tikzcd}
		E_n\arrow[d,"i_1"]\arrow[r,"\phi_n"] &F_{n}\arrow[d]\\
		E_n\times I\arrow[r] &M_{\phi_n}
		\end{tikzcd}
		\]
		and hence the morphism $F_{n}\to M_{\phi_n}$ is an acyclic cofibration with a left inverse induced by $id_{F_{n}}:F_{n}\to F_{n}$ and $E_n\times I\to E_n\xrightarrow{\phi_n} F_n$. As $T$ is a left adjoint it preserves pushouts and hence $(T(\phi_n), H_n)$ induces a map 
		\[T(M_{\phi_n})\to F_{n+1}\to M_{\phi_{n+1}} .\]
		This gives us a spectrum $Cyl(\phi)=(M_{\phi_0},M_{\phi_1},\ldots)$ with levelwise weak equivalences $\mathbf{E}\to Cyl(\phi)$ and $\mathbf{F}\to Cyl(\phi)$ given by 
		\[
		E_n\xrightarrow{i_0} E_n\times I\to M_{\phi_n}\quad\text{and}\quad F_n \to M_{\phi_n},
		\]
		where $E_n\times I\to M_{\phi_n}$ is a weak equivalence by the 2-out-of-3 property. Hence $\mathbf{E}$ and $\mathbf{F}$ are isomorphic in $H(Sp^{\mathbb{N}}(\mc{C},T))$. 	
	\end{proof}   
	These two lemmas show that a naive $T$-spectrum defines a unique object in $H^l(Sp^{\mathbb{N}}(\mc{C},T))$ up to isomorphism.
	\begin{theorem}\label{A:main}
		For any naive $T$-spectrum $({E_n,e_n})_{n\in \mathbb{N}}$, there exists a $T$-spectrum $\mathbf{E}\in Sp^{\mathbb{N}}(\mc{C},T)$, with assembly maps $f'_n:T(\mathbf{E}_n)\to \mathbf{E}_{n+1}$, such that
		\begin{enumerate}
			\item $\naive(\mathbf{E})\iso (E_n,e_n)_{n\in \mathbb{N}}$ in $SH(\mc{C},T)$;
			\item any other $\mathbf{E}'\in Sp^{\mathbb{N}}(\mc{C},T)$ satisfying the above condition is isomorphic to $\mathbf{E}$ in $H^l(Sp^{\mathbb{N}}(\mc{C},T))$.
		\end{enumerate} 
		In particular every naive $T$-spectrum defines a unique object in $SH(\mc{C},T)$ up to isomorphism.
	\end{theorem}
	\begin{proof}
		Statement (2) follows from statement (1) and Lemma~\ref{l:iso}. Statement (1) is just a reformulation of Lemma~\ref{l:surj}. The last sentence follows because every levelwise weak equivalence is a stable equivalence.
	\end{proof}
	
	\begin{remark}
		Using the projective model structure is not necessary; all we need is a model structure on $Sp^\mathbb{N}(\mathcal{C},T)$ where the weak equivalences are precisely the levelwise weak equivalences and every bifibrant object is also levelwise bifibrant. Therefore the injective model structure works as well.
	\end{remark}

	\section{Thom spaces}
	The Thom space construction preserves $\mathbb{A}^1$-weak equivalences for sufficiently well behaved schemes. Let $S$ be a scheme which is ind-smooth over a Dedekind ring $k$ with perfect residue fields. 
	\begin{theorem}\label{thm:thom-iso}
		Let $X\in Sm_{S}$ be a smooth $S$-scheme with a rational point $x:S\to X$. For any $\mathbb{A}^1$-equivalence of pointed smooth schemes $f:(Y,y)\to (X,x)$ and any vector bundle $E\to X$ of constant rank $n$, the induced map of Thom spaces $Th(f^*E)\to Th(E)$ is an $\mathbb{A}^1$-equivalence.
	\end{theorem}
	To prove this we use the following lemma.
	\begin{lemma}\label{lem:thom-iso}
		Let $E\to B$ be a principal $GL_n$-bundle over $S$. Given a point $b:S\to B$, the diagram 
		\[GL_n\to E\to B \]
		coming from the pullback
		\[ 
		\begin{tikzcd}
		GL_n\arrow[r]\arrow[d] &E\arrow[d]\\
		S\arrow[r,"b"] &B
		\end{tikzcd}
		\]
		is an $\mathbb{A}^1$-local fiber sequence. Furthermore, for any scheme $F$ with a $GL_n$-action $\sigma:GL_n\times F\to F$ the induced diagram
		\[F\to E\times_\sigma F\to B \]
		is an $\mathbb{A}^1$-local fiber sequence.
	\end{lemma}
	\begin{proof}
		Note that for any locally trivial bundle $P\to X$ with fiber $F$ and any point $x$ in $X$ the pullback diagram
		\[F\to P\to X\]
		is a simplicial fiber sequence (taking stalks gives a fiber sequence of simplicial sets). For any smooth $S$-scheme $B$ there is a sequence of bijections
		\[Vect_n(B)\cong P_{Nis}(B,GL_n)\cong P_{Nis}(B,GL_n)\cong Hom_{H^s(S)}(B,BGL_n) \] 
		where, $P_\tau(B,GL_n)$ is the set of $\tau$-locally trivial $GL_n$ bundles, by \cite[Ex.2.3.4]{AHW18} and \cite[Prop.4.1.15]{MV99}. This implies that the map $E\to B$ is a pullback of the $GL_n$-bundle $E_{Nis}GL_n\to B_{Nis}GL_n$ where $B_{Nis}GL_n$ is a Nisnevich fibrant replacement. As every vector bundle (and hence every $GL_n$-torsor) is Zariski locally trivial, $BGL_n$ satisfies Nisnevich descent and hence we have a bijection  
		\[\pi_0(B_{Nis}GL_n(X))\cong\pi_0(BGL_n(X)) \]
		for any $X\in Sm_S$. By \cite[Thm.~5.2.3]{AHW18}, the set of rank $n$ vector bundles $Vect_n(-)$ is $\mathbb{A}^1$-invariant for affine schemes over $S$ and hence we have
		\[Vect_n(X)\cong \pi_0(B_{Nis}GL_n(X))\cong\pi_0(BGL_n(X))\]
		for $X$ affine. By \cite[Thm.~2.2.5]{AHW18}
		\[G\to E_{Nis}G\to B_{Nis}G\]
		is an $\mathbb{A}^1$-local fiber sequence hence by \cite[Prop.~2.3]{We11}
		\[GL_n\to E\to B\]
		is an $\mathbb{A}^1$-local fiber sequence. For any scheme $F$ with a $GL_n$-action we can show that 
		\[F\to E_{Nis}(GL_n)\times_\sigma F\to B_{Nis}(GL_n) \]
		is an $\mathbb{A}^1$-local fiber sequence along the lines of \cite[Prop.~5.1]{We11}. The simplicial fiber sequence
		\[F\to E\times_\sigma F\to B \]
		is a pullback of the universal sequence and hence is also $\mathbb{A}^1$-local. 
	\end{proof}
	\begin{proof}[Proof of Theorem~\ref{thm:thom-iso}]
		Given any vector bundle $E\to X$ of rank $n$, the 2 out of 3 property implies that an $\mathbb{A}^1$-weak equivalence $f:Y\to X$ induces an $\mathbb{A}^1$-weak equivalence $f^*E\to E$.  The complement of the zero section $E-X\to X$ is a locally trivial bundle with fiber $\mathbb{A}^n-0$. The fiber sequence
		\[\mathbb{A}^n-0\to E-X\to X \]
		is obtained by twisting the $GL_n$-torsor associated to the vector bundle $E\to X$ by the standard $GL_n$-action on $\mathbb{A}^n-0$ and is hence an $\mathbb{A}^1$-local fiber sequence by Lemma~\ref{lem:thom-iso}. 
                Thus the pullback of $E-X$ along an $\mathbb{A}^1$-equivalence $f:Y\to X$ induces an $\mathbb{A}^1$-equivalence $f^*E-Y\to E-X$. 
                We therefore have an equivalence of cofibration sequences 
		\[
		\begin{tikzcd}
		f^*E-Y\arrow[r]\arrow[d,"\sim"] &f^*E\arrow[d,"\sim"]\arrow[r] &Th(f^*E)\arrow[d,"\sim"]\\
		E-X\arrow[r] &E\arrow[r] &Th(E)
		\end{tikzcd}
		\]
		giving the desired $\mathbb{A}^1$-equivalence of Thom spaces.
	\end{proof}
	\end{appendices}
	\bibliography{bibliography1}
	\bibliographystyle{alphaurl}
\end{document}